\documentclass[12pt]{article}
\usepackage{amsmath,amssymb,amsthm,amscd}

\usepackage{graphics, graphpap}
\newtheorem{theorem}{Theorem}[section]
\newtheorem{corollary}[theorem]{Corollary}
\newtheorem{lemma}[theorem]{Lemma}

\theoremstyle{definition}
\newtheorem{definition}[theorem]{Definition}
\newtheorem{remark}[theorem]{Remark}

\def\Hom{\operatorname{Hom}}
\def\p{\frak p}
\def\m{\frak m}
\def\q{\frak q}

\def\Spec{\operatorname{Spec}}

 \def\Hom{\operatorname{Hom}}
\def\dim{\operatorname{dim}}
\def\Ext{\operatorname{Ext}}
\def\Im{\operatorname{Im}}

\def\depth{\operatorname{depth}}
\def\f-depth{\operatorname{f-depth}}
\def\gdepth{\operatorname{gdepth}}
\def\Supp{\operatorname{Supp}}
\def\Ass{\operatorname{Ass}}
\def\ann{\operatorname{ann}}

\def\Ker{\operatorname{Ker}}

\def\inf{\operatorname{inf}}

\def\sr{\rightarrow}

\newcommand{\fM}{\ensuremath{\mathcal M}}
\newcommand{\fR}{\ensuremath{\mathcal R}}

\begin{document}

\title{Asymptotic stability of certain sets of associated prime ideals of local cohomology modules}
\author{Nguyen Tu Cuong$^a$, Nguyen Van Hoang$^b$ and Pham Huu Khanh$^c$}
\date{Institute of Mathematics\\18 Hoang Quoc Viet Road,  10307  Hanoi, Vietnam}
\maketitle

{\begin{quote} \normalsize{\bf Abstract}{\footnote{{{\it Key words
and phrases}: associated prime,  local cohomology, depth, filter depth, generalized depth }  \hfill\break {{\it 2000 Subject Classification}:
13D45, 13E05} \hfill\break {This work is supported in part by the
National Basis Research Programme in Natural Science of
Vietnam.}\hfill\break {$^a$ E-mail: ntcuong@math.ac.vn}\hfill\break
{$^b$  E-mail: nguyenvanhoang1976@yahoo.com}\hfill\break {$^c$
E-mail: huukhanh75@yahoo.com}}. Let $(R,\m)$ be a Noetherian local
ring  $I, J$ two ideals of $R$ and $M$  a finitely generated $R-$module. It is first shown that for $k\geq -1$  the integer $r_k = \depth_k(I,J^nM/J^{n+1}M)$, it is the length of a maximal $(J^nM/J^{n+1}M)-$sequence in dimension $>k$ in $I$ defined by M. Brodmann and L. T. Nhan \cite{BN}, becomes for large $n$ independent of $n$. Then we prove in this paper that  the sets $\bigcup_{j\le r_k}\Ass_R(H^j_I(J^nM/J^{n+1}M))$ with $k=-1$ or $k=0$, and $\bigcup_{j\le r_1}\Ass_R(H^j_I(J^nM/J^{n+1}M))\cup\{\m\}$ are stable for large $n$. We also obtain similar results for modules $M/J^nM$.}
\end{quote}

\section{Introduction} Let $(R,\m)$ be a Noetherian local ring, $I, J$ two ideals of $R$ and $M$  a finitely generated $R-$module. This paper is concerned the asymptotic stability of sets of associated prime ideals  $\Ass_R(H^j_I(J^nM/J^{n+1}M))$ and $\Ass_R(H^j_I(M/J^nM))$. There are two starting points for our investigation as follows.
\medskip

Firstly, In 1979, M. Brodmann \cite{Br1} had proved that the sets $\Ass_R(J^nM/J^{n+1}M)$ and $\Ass_R(M/J^nM)$ are stable for large $n$. Based on this result he showed in \cite[Theorem 2 and Proposition12]{Br2} that the integers $\depth(I,J^nM/J^{n+1}M)$ and $\depth(I,M/J^nM)$ take constant values for large $n$. Recently, M. Brodmann and L.T. Nhan introduced the notion of $M-$sequence in dimension $>k$  in \cite{BN}: Let $k$ be an integer $k\ge -1$. A sequence $x_1,\ldots,x_r$ of elements of $\m$ is called an $M-$sequence
in dimension $> k$ if $x_i\notin\p$ for all $\p\in\Ass_R(M/(x_1,\ldots,x_{i-1})M)$ with $\dim(R/\p)> k$ and
all $i=1,\ldots, r$. They also showed that every maximal $M-$sequence in dimension $>k$ in $I$ have the same length, and the common length is the least integer $i$ such that there exists $\p\in\Supp(H^i_I(M))$ with $\dim(R/\p)>k$. For convenience, we denote this common length by $\depth_k(I,M)$. It should be mentioned that  $\depth_{-1}(I,M)$ is just  $\depth(I,M)$ the length of a maximal $M$-sequence in $I$, $\depth_0(I,M)$ is the filter depth $\f-depth(I,M)$ with respect to the ideal $I$ defined by R. L${\rm\ddot u}$ and Z. Tang \cite{LT}, and $\depth_1(I,M)$ is the generalized depth $\gdepth(I,M)$ with respect to the ideal $I$ defined by L.T. Nhan \cite{Nh}. The first result of this paper is a generalization of Brodmann's Theorem.

\begin{theorem}\label{L3} For all $k\ge -1$,  $\depth_k(I,J^nM/J^{n+1}M)$ and $\depth_k(I,M/J^nM)$ take constant values $r_k$ and $s_k$, respectively,  for large $n$ .
\end{theorem}

Secondly, in 1990, C. Huneke \cite[Problem 4]{Hun} asked whether the set of associated primes of $H^i_I(M)$ is finite for all generated modules $M$ and all ideals $I$. Affirmative answers were given by G. Lyubeznik  \cite{Luy} and Huneke-R.Y. Sharp \cite{HuSh} for equicharacteristic regular local rings. Although, A. Singh  \cite{AS} and M. Katzman  \cite{Kat} provided  examples of finitely generated modules having some local cohomology with infinite associated prime ideals, the problem is still true in many situations ( see \cite{BF} \cite{KhS}, \cite{KhS1}, \cite{Mar}, \cite{Nh}...). Especially, it was proved in \cite{KhS1} and \cite{Nh} that $\Ass_R(H^j_I(M))$ is finite for all $j\leq \depth_1(I,M)$.  From this and  Theorem \ref{L3} we obtain for all $j\leq r_1=\depth_1(I,J^nM/J^{n+1}M)$ and $i\leq s_1=\depth_1(I,M/J^nM)$ that the sets $\Ass_R(H^{j}_I(J^nM/J^{n+1}M)$ and $\Ass_R(H^{i}_I(M/J^nM)$ are finite for large $n$. So it is natural to ask the following question:\\
{\it Are  the sets $\Ass_R(H^{j}_I(J^nM/J^{n+1}M))$ and $\Ass_R(H^{i}_I(M/J^nM))$ for  $j\leq r_1$ and $i\leq s_1$   stable for  large $n$? }
\medskip

The main result of this paper is the following theorem which shows that a weakening of the above questions may have a positive answer. 

\begin{theorem}\label{MTh} Let $k\ge -1$ be an integer, and  $r_k$,  $s_k$ as in Theorem \ref{L3}. Then the following statements are true.
\begin{itemize}
\item[\rm (i)] $\Ass_R(H^{r_{-1}}_I(J^nM/J^{n+1}M))$ and $\Ass_R(H^{s_{-1}}_I(M/J^nM))$ are stable for large $n$.
\item[\rm (ii)] $\displaystyle\bigcup_{j\le r_0}\Ass_R(H_I^j(J^nM/J^{n+1}M))$ and $\displaystyle\bigcup_{i\le s_0}\Ass_R(H_I^i(M/J^nM))$ are stable for large $n$.
\item[\rm (iii)] $\displaystyle\bigcup_{t\le j}\Ass_R(H_I^t(J^nM/J^{n+1}M))\cup \{\m\}$ and $\displaystyle\bigcup_{t\le i}\Ass_R(H_I^t(M/J^nM))\cup \{\m\}$ for all $j\le r_1$ and $i\le s_1$ are stable for large $n$.
\end{itemize}
\end{theorem}
Notice here that since $r_{-1}\leq r_0\leq r_1$ as well as $s_{-1}\leq s_0\leq s_1$ and the modules $H^{j}_I(J^nM/J^{n+1}M),$ $ H^{i}_I(M/J^nM)$ are either vanished or Artinian for all  $j< r_0$ and $i< s_0$ by \cite{LT}, the  question is trivial at these $j$, $i$. So the conclusion (i) of Theorem \ref{MTh} is the first non-trivial answer to the above question at $j=r_{-1}$  and $i= s_{-1}$ when $r_{-1}=r_0$  and $ s_{-1}=s_0$. Moreover, the conclusion (ii) (the conclusion (iii), respectively)  of Theorem \ref{MTh} says that except the maximal ideal $\m$ (a finite set, respectively)
the sets in the question above are stable for all $j\leq r_0$ and $i\leq s_0$ ($j\leq r_1$ and $i\leq s_1$,  respectively).  
\medskip

This paper is divided into 5 sections. In Section 2 we study some properties concerning generalizations of the depth of a finite generated module over a local ring. Theorem \ref{L3} will be  also showed in this section.  Sections 3, 4 and 5 are devoted to prove the statements (i), (ii) and (iii), respectively, of Theorem \ref{MTh}.

\section{Asymptotic stability of generalizations of depth}
Throughout this paper, $(R,\m)$ is a Noetherian local ring with the maximal ideal $\m$, $I, J$ two ideals of $R$ and $M$ is  a finitely generated $R-$module.
\begin{definition} (\cite[Definition 2.1]{BN}) Let $k\ge -1$ be an integer. A sequence
$x_1,\ldots,x_r$ of elements of $\m$ is called an {\it $M-$sequence
in dimension $> k$} if $x_i\notin\p$ for all
$\p\in\Ass_R(M/(x_1,\ldots,x_{i-1})M)$ with $\dim(R/\p)> k$ and
all $i=1,\ldots, r$.
\end{definition}

It is clear that $x_1,\ldots ,x_r$ is an $M-$sequence in dimension $>-1$ if and only if it is a regular sequence of $M$; and $x_1,\ldots ,x_r$ is an $M-$sequence in dimension $>0$ if and only if it is a filter regular sequence of $M$  introduced by P. Schenzel, N. V. Trung and the first author in \cite{Cst}. Moreover, $x_1,\ldots ,x_r$ is an $M-$sequence in dimension $>1$ if and only if it is a generalized regular sequence of $M$ defined by L. T. Nhan in \cite{Nh}.

\begin{remark}\label{R1} (i) Let $k$ be a non negative integer.  Assume that $\dim(M/IM)> k$. Then any $M-$sequence in dimension $> k$ in $I$ is of finite length, and all maximal $M-$sequence in dimension $>k$ in $I$ have the same length which is equal to the least integer $i$ such that there exists $\p\in\Supp(H^i_I(M))$ with $\dim(R/\p)>k$ (\cite[Lemma 2.4]{BN}). We denote in this case by $\depth_k(I,M)$ the length of a maximal $M-$sequence in dimension $>k$ in $I.$ Moreover, if $x_1,\ldots,x_r$ is a maximal $M-$sequence in dimension $>k$ in $I$, then $x_1,\ldots,x_r$ is a part of system of parameters of $M$, and thus $\depth_k(I,M)\le\dim(M)-\dim(M/IM)$. Note that $\depth_{-1}(I,M)$ is  the usual depth $\depth(I,M)$ of $M$ in $I$,  $\depth_0(I,M)$ is the filter depth $\f-depth(I,M)$ of $M$ in $I$  denoted by L${\rm\ddot u}$ and Tang in \cite{LT} and  $\depth_1(I,M)$ is just the generalized depth $\gdepth(I,M)$ of $M$ in $I$ defined by Nhan \cite{Nh}. 
\medskip

\noindent(ii) If $\dim(M/IM)\leq k$, then we can choose an $M-$sequence in dimension $> k$ in $I$ of length $r$ for every positive integer  $r$, and  we set in this case that $\depth_k(I,M)=\infty$. 
\end{remark}

Let  $S$ be a subset of $\Spec(R)$ and  $i\ge 0$ an integer, we set 
$$S_{\ge i}=\{\p\in S| \dim (R/\p)\ge i\}\text{ and }S_{>i}=\{\p\in S| \dim (R/\p)>i\}.$$

\begin{lemma}\label {LL0} Let $k\ge -1$ be an integer. Then 
$$
\begin{aligned}
\depth_k(I,M)&=\inf\{j\mid\dim(\Ext^j_R(R/I,M))>k\}\notag\\
&=\inf\{\depth_{k-i}(I_\p,M_\p)\mid\p\in\Supp(M/IM)_{\ge i}\}\notag
\end{aligned}
$$
for all $0\leq i\leq k+1$, where we use the convenience that $\inf(\emptyset)=\infty.$
\end{lemma}

\begin{proof} Since $\depth_k(I,M)=\infty$ if and only if $\dim(M/IM)\leq k$, the statement is clear for the case $\depth_k(I,M)=\infty$. Assume that  $r=\depth_k(I,M)$ is a non negative integer. By \cite[Lemma 2.4]{BN}, we get $$r=\inf\{i\mid\exists\p\in\Supp(H^i_I(M)),\dim(R/\p)>k\}.$$
Moreover, by \cite[Lemma 2.8]{CH1}, we have 
$$\bigcup_{j\le l}\Supp(H^j_I(M))=\bigcup_{j\le l}\Supp(\Ext^j_R(R/I,M))$$ for all $l\ge 0$. It follows that $r=\inf\{j\mid\dim(\Ext^j_R(R/I,M))>k\}.$
\medskip
To prove the second equality, let $x_1,\ldots,x_r\in I$ be an $M-$sequence in dimension $>k$, and $i\in\{0,\ldots,k+1\}$. For each $\p\in\Supp(M/IM)_{\ge i}$,   $x_1/1,\ldots,x_r/1$ is an $M_\p-$sequence in dimension $>k-i$ in $I_\p$; and thus $r\le\depth_{k-i}(I_\p,M_\p)$. On the other hand, by the first equality, there exists $\q\in\Supp(\Ext^r_R(R/I,M))$ with $\dim(R/\q)>k$. Hence, we can choose a prime ideal $\p'$ such that  $\p'\supseteq\q$, $\dim(R/\p')=i$ and $\dim(R_{\p'}/\q R_{\p'})>k-i$. Therefore $\q R_{\p'}\in\Supp_{R_{\p'}}(\Ext^r_{R_{\p'}}(R_{\p'}/I_{\p'},M_{\p'}))_{>k-i}$. It follows that $r\ge\depth_{k-i}(I_{\p'},M_{\p'})$,  and so $r=\inf\{\depth_{k-i}(I_\p,M_\p)\mid\p\in\Supp(M/IM)_{\ge i}\}$ as required.
\end{proof} 

Let  $\fR=\oplus_{n\ge 0}R_n$ be finitely generated a standard graded algebra over $R_0=R$, and $\fM=\oplus_{n\ge 0}M_n$ is a finitely generated graded $\fR-$module. For the convenience, we  use $N_n$ to denote from now on  either  the $R-$module $M_n$ or $M/J^nM$. Then Brodmann's Theorem on the asymptotic stability of associated prime ideals \cite{Br1} (see also \cite[Theorem 3.1]{Me}) can be stated as follows. 
 
\begin{lemma} \label{L2} The set  $\Ass_R(N_n)$ is stable for large $n$.
\end{lemma}

\begin{lemma}\label{L221} Let $k\ge -1$ and $r\ge 1$ be integers. If $\dim(\Ext^i_R(R/I,N_n))\le k$ for infinitely many $n$ and all $i<r$. Then there always  exists a sequence of $r$ elements  $x_1,\ldots,x_r\in I$ which is  an $N_n-$sequence in dimension $>k$  for all large $n$.
\end{lemma}

\begin{proof} Assume that $T$ is an infinite subset of the set of integers such that for all $n\in T$ and  $i<r$ we have $\dim(\Ext^i_R(R/I,N_n))\le k$ . We proceed by induction on $r$. If $r=1$, then $\dim(\Hom(R/I,N_n))\le k$ for all $n\in T$. It implies  $I\nsubseteq\p$ for all $\p\in\Ass_R(N_n)_{>k}$ and all $n\in T$. Since $\Ass_R(N_n)$ is stable for large $n$ by Lemma \ref{L2},  there exists an integer $a\in T$ such that $I\nsubseteq\p$ for all $\p\in\Ass_R(N_n)_{>k}$ and all $n\ge a$. Hence, we can choose an element $x_1\in I$ which is an $N_n-$sequence in dimension $>k$ for all $n\ge a$. Assume that $r>1$. Note that $\dim(0:_{N_n}x_1)\le k$ and $\dim(\Ext^i_R(R/I,N_n))\le k$ for all $n\in T$ with $n\ge a$ and all $i<r$. From this, by using the following exact sequences
$$\Ext^i_R(R/I,N_n)\sr\Ext^i_R(R/I,N_n/(0:_{N_n}x_1))\sr\Ext^{i+1}_R(R/I,(0:_{N_n}x_1))$$
and 
$$\Ext^{i-1}_R(R/I,N_n)\sr\Ext^{i-1}_R(R/I,N_n/x_1N_n)\sr\Ext^i_R(R/I,N_n/(0:_{N_n}x_1)),$$
we obtain that $\dim(\Ext^i_R(R/I,N_n/x_1N_n))\le k$ for all $n\in T$ with $n\ge a$ and all $i<r-1$. Hence, by  inductive hypothesis, there exists an $N_n/x_1N_n-$sequence in dimension $>k$ $x_2,\ldots,x_r\in I$ for all  $n\ge a$ for some $a\in T$ as required.
\end{proof}

\begin{proof}[\bf Proof of Theorem \ref{L3}] By Lemma \ref{L2}, there is an integer $u>0$ such that $d=\dim(N_n)$ and $d'=\dim(N_n/IN_n)$ for all $n\ge u$. If $d'\leq k$, then $\depth_k(I,N_n)=\infty$ for $n\ge u$, and the conclusion follows.\\
Assume that $d'>k$. Then $0\le\depth_k(I,N_n)\leq d-d'$ for all $n\ge u$ by Remark \ref {R1}, (i). Thus there exists an infinite subset $T$ of $\Bbb Z$, and an integer $r\in\{0,\ldots,d-d'\}$ such that 
$$r=\inf\{i\mid\dim(\Ext^i_R(R/I,N_n))>k\}$$
for all $n\in T$. We will show that $\depth_k(I,N_n)=r$ for all $n$ large. Indeed, if $r=0$, then $\dim(\Hom(R/I,N_n))>k$ for all $n\in T$. Since $\Ass_R(\Hom(R/I,N_n))=\Ass_R(N_n)\cap V(I)$,  we get by Lemma \ref{L2} that the set $\Ass_R(\Hom(R/I,N_n))$ is stable for all $n\ge a$ for some $a\in T$. It follows that $\depth_k(I,N_n)=0$ for all $n\ge a$. If $r\ge 1$, then we get by Lemma \ref{L221} that $\depth_k(I,N_n)\ge r$ for all $n\ge v$ for some $v>0$. Therefore $\dim(\Ext^i_R(R/I,N_n))\le k$ for all $n\ge v$ and all $i<r.$ Assume that $\dim(\Ext^r_R(R/I,N_n))\le k$ for all $n$ in an infinite subset $S$ of $\Bbb Z$. Then there exists by Lemma \ref{L221}  an integer $b\ge v$ such that $\depth_k(I,N_n)\ge r+1$, and so $\dim(\Ext^r_R(R/I,N_n))\le k$  for all $n\ge b$ by Lemma \ref{LL0}. This contradicts the definition  of $r$. Thus there exists an enough large integer $c$ such that $\dim(\Ext^r_R(R/I,N_n))>k$ and $\dim(\Ext^i_R(R/I,N_n))\le k$ for all $n\ge c$ and $i<r$. Therefore by Lemma \ref{LL0} again we get $\depth_k(I,N_n)=r$ for all $n\ge c$ .
\end{proof}

In view of Theorem \ref{L3}, if $r=\depth_k(I,N_n)$ for all $n$ large, we call that  $r$ is  the  eventual value of $\depth_k(I,N_n)$. By Lemma \ref{L221} and the proof of Theorem \ref{L3}, we get immediately the following consequence.

\begin{corollary} Let $k\ge -1$ be an integer and $r$ the eventual value of $\depth_k(I,N_n)$. Then 
$$r=\inf\{i\mid\dim(\Ext^i_R(R/I,N_n))>k\text{ for infinitely many }n\}.$$
\end{corollary}

\begin{lemma}\label{C22} Let $k\ge -1$ be an integer. Assume that $r$ and $s$ are the eventual values of $\depth_k(I,J^nM/J^{n+1}M)$ and $\depth_k(I,M/J^nM)$, respectively. Then we have $r\ge s$.
\end{lemma}

\begin{proof} Let $n>0$. Set 
$r(n)=\depth_k(I,J^nM/J^{n+1}M)\text{ and }s(n)=\depth_k(I,M/J^nM).$
By the short exact sequence 
$$0\sr J^nM/J^{n+1}M\sr M/J^{n+1}M\sr M/J^nM\sr 0,$$
we get the long exact sequence
$$
\begin{aligned}
\cdots\sr\Ext^{j-1}_R(R/I,M/J^nM)\sr\Ext^j_R&(R/I,J^nM/J^{n+1}M)\\
&\sr\Ext^j_R(R/I,M/J^{n+1}M)\sr\cdots.
\end{aligned}
$$
Therefore for any $j<\min\{s(n)+1,s(n+1)\}$ we get by Lemma \ref{LL0} that $$\dim(\Ext^j_R(R/I,J^nM/J^{n+1}M))\le k.$$
It implies that $r(n)\ge\min\{s(n)+1,s(n+1)\}$ by Lemma \ref{LL0}. Note that, by the hypothesis of $r$ and $s$, there exists an integer $a>0$ such that $r=r(n)$ and $s=s(n)$ for all $n\ge a$. It follows that $r\ge s$.
\end{proof}

Recall that a sequence $x_1,\ldots ,x_r\in I$ is called an {\it $I-$filter regular sequence} of $M$ if $x_i\notin\p$ for all $\p\in\Ass_R(M/(x_1,\ldots ,x_{i-1})M)\setminus V(I)$ and all $i=1,\ldots ,r$, where $V(I)$ is the set of all prime ideals of $R$ containing $I$. 

\begin{lemma}\label{L4} Let $k\geq -1$ be an integer and $r$  the eventual value of $\depth_k(I,N_n)$.
Assume that $1\leq r<\infty$. Then there is a sequence $x_1,\ldots ,x_r$ in $I$ which is at the same time a permutable $N_n-$sequence in dimension $> k$ and a permutable
$I-$filter regular sequence of $N_n$ for all large $n$.
\end{lemma}

\begin{proof} Let $l$ be an integer that such $1\leq l\leq r$. We will show by induction on $l$ the existence of a sequence of $l$ elements in $I$ satisfying the conclusion of the lemma. If $l=1$, then by Lemma \ref{L2} we may assume that $\Ass_R(N_n)$ is stable for all $n\ge t$ for some integer $t>0$. Since $r\geq 1$, $I\nsubseteq\p$
for all $\p\in\Ass_R(N_n)_{> k}$ and all $n\ge t$. From this, by the Prime Avoidance Theorem, $I\nsubseteq\p$ for all $\p\in \Ass_R(N_n)_{> k}\cup(\Ass_R(N_n)\setminus V(I))$ and all $n\ge t$. Therefore, there is an
 element $x_1\in I$ satisfying our lemma.
\medskip

Assume that $1<l\leq r$ and there already exists a sequence $x_1,\ldots,x_{l-1}\in I$ which is at the same time a permutable $N_n-$sequence in dimension $>k$ and a permutable $I-$filter regular sequence of $N_n$ for all $n\ge t$ for some integer $t>0$. Let $\Omega$ be the set of all subsets of $\{1,\ldots,l-1\}$. 
Hence, since $\Omega$ is finite, we get by Lemma \ref{L2} that the set 
$$T=\bigcup_{\Lambda\in\ \Omega}\big(\Ass_R(N_n/\Sigma_{j\in\Lambda}x_jN_n)_{> k}\cup(\Ass_R(N_n/\Sigma_{j\in\Lambda}x_jN_n)\setminus V(I))\big)$$
is finite and stable for all $n$ large. Hence there exists an
element $x_l\in I$ such that $x_l\notin\p$ for all $\p\in T$
by the Prime Avoidance Theorem. Then with the same method that used  in the proof of Brodmann-Nhan in \cite[Proposition 2.5]{BN}, we can show that $x_1,\ldots ,x_l$ is at the same time a permutable $N_n-$sequence in dimension $> k$ and a permutable
$I-$filter regular sequence of $N_n$ for all large $n$.
\end{proof}

Finally, we recall the following fact which will be used in the sequel.

\begin{lemma}\label{LL1} {\rm (see \cite[1.2]{KhS}, \cite[3.4]{NS})}
If $x_1,\ldots,x_r$ is an $I-$filter regular sequence of $M$ then we have
$$
H^j_I(M)=
\begin{cases}
 H^j_{(x_1,\ldots,x_r)}(M)& \text{ if } j<r\\
 H^{j-r}_I(H^r_{(x_1,\ldots,x_r)}(M))& \text{ if } j\geq r.
\end{cases}
$$
\end{lemma}

\section{Proof of Theorem \ref{MTh}, (i)}
 In the rest of this paper, instead of $\depth_{-1}(I,M)$, $\depth_0(I,M)$ and $\depth_1(I,M)$ we denote,  by tradition,  by $\depth(I,M)$,   $\f-depth(I,M)$ and $\gdepth(I,M)$, respectively.
We first recall a known fact on local cohomology modules.

\begin{lemma}\label{L31} {\rm (cf. \cite[Theorem 2.4]{CH})}  Let $r=\depth(I, M)$. Assume that $1\le r<\infty$ and $x_1,\ldots,x_r$ a regular sequence of $M$ in $I$. Then we have
$$\Ass_R(H^r_I(M))=\Ass_R(M/(x_1,\ldots,x_r)M)\cap V(I).$$
\end{lemma}

The  statement (i) of Theorem \ref{MTh} is then an immediate consequence of following theorems.
\medskip

\begin{theorem}\label{KL} Let $r$ be the eventual value of $\depth(I,M_n)$. Then the set $\Ass_R(H^r_I(M_n))$ is stable for large $n$.
\end{theorem}
\begin{proof} If $r=\infty$ then $\Ass_R(H^r_I(M_n))=\emptyset$ for all large $n$. If $r=0$ then $\Ass_R(H^0_I(M_n))=\Ass_R(M_n)\cap V(I)$ is stable for large $n$ by Lemma \ref{L2}. It remains to consider the case $1\le r<\infty$. Then we get from Theorem \ref{L3}, Lemma \ref{L2} and Lemma \ref{L4} that there exists a positive integer $a$ such that for all $n\ge a$ the following are true: (i) $r=\depth(I,M_n)$;  (ii) there exists a sequence $x_1,\ldots,x_r\in I$ which is a regular sequence of $M_n$; and (iii) the set  $\Ass_R(M_n/(x_1,\ldots,x_r)M_n)$ is independent of  $n$.  Therefore, we get by Lemma \ref{L31} that 
$$\Ass_R(H^r_I(M_n))=\Ass_R(M_n/(x_1,\ldots,x_r)M_n)\cap V(I)$$
for all $n\ge a$, and so that $\Ass_R(H^r_I(M_n))$ is stable for large $n$.  
\end{proof}

\begin{theorem}\label{T41} Let $s$ be the eventual value of $\depth(I,M/J^nM)$. Then the set $\qquad$
$\Ass_R(H^s_I(M/J^nM))$ is stable for large $n$.
\end{theorem}

\begin{proof} Let $r$ be the eventual value of $\depth(I,J^nM/J^{n+1}M)$. Then $r \ge s$ by Lemma \ref{C22}. Moreover, we can choose an integer $a>0$ such that $r=\depth(I,M/J^nM)$ and $s=\depth(I,J^nM/J^{n+1}M)$ for all $n\ge a$. From the short exact sequence
$$0\sr J^nM/J^{n+1}M\sr M/J^{n+1}M\sr M/J^nM\sr 0,$$ we
get the following long exact sequence
$$
\begin{aligned}
\cdots\sr H^{j-1}_I(M/J^nM)\sr H^j_I(J^n&M/J^{n+1}M)\sr\\
&\sr H^j_I(M/J^{n+1}M)\sr H^j_I(M/J^nM)\sr\cdots
\end{aligned}
$$
for all $n>0$. We  consider three cases as follows.
\medskip

\noindent{\it Case 1}:  $r-2\ge s$. By the long exact sequence, we have the following isomorphisms
$$H^{s}_I(M/J^{n+1}M)\cong H^{s}_I(M/J^nM)$$
 for all $n\geq a$. Thus
$\Ass_R(H^{s}_I(M/J^nM))$ is stable for large $n$.
\medskip

\noindent{\it Case 2}: $r-1=s$. By the long
exact sequence above, we get the  exact sequence
$$0\sr H^{s}_I(M/J^{n+1}M)\sr H^{s}_I(M/J^nM)\sr H^{s+1}_I(J^nM/J^{n+1}M)$$
for all $n\ge a$. So $\Ass_R(H^{s}_I(M/J^{n+1}M))\subseteq\Ass_R(H^{s}_I(M/J^nM))$
for all $n\ge a$. It follows that 
$\Ass_R(H^{s}_I(M/J^nM))$ is stable for large $n$, since
$\Ass_R(H^{s}_I(M/J^aM))$ is a finite set by Lemma \ref{L31}.
\medskip

\noindent{\it Case 3}:  $r=s$. By the long exact sequence, we have the following exact sequence 
$$0\sr H^{s}_I(J^nM/J^{n+1}M)\sr H^{s}_I(M/J^{n+1}M)\sr H^{s}_I(M/J^nM)$$
for all $n\ge a$. There exists by Theorem \ref{KL} an integer $b\ge a$ such that 
$$\Ass_R(H^{s}_I(J^nM/J^{n+1}M)=\Ass_R(H^{s}_I(J^bM/J^{b+1}M)$$
for all $n\ge b$. Set $X=\Ass_R(H^s_I(J^bM/J^{b+1}M))$, then we get by the previous  exact sequence that 
$$X\subseteq \Ass_R(H_I^{s}(M/J^{n+1}M)) \subseteq
\Ass_R (H_I^{s}(M/J^nM))\cup X$$ for all $n\ge b$. Hence, for any $n\ge b$, we have
$$\Ass_R (H_I^{s}(M/J^{n+2}M))\subseteq \Ass_R
(H_I^{s}(M/J^{n+1}M))\cup X=\Ass_R(H_I^{s}(M/J^{n+1}M)).$$ 
Since $\Ass_R(H_I^s(M/J^{b+1}M))$ is finite by Lemma \ref{L31}, it follows that $\Ass_R(H_I^s(M/J^nM))$ is stable for large $n$, this finishes the proof of the theorem.
\end{proof}

\section{Proof of Theorem \ref{MTh}, (ii)}

We first need a characterization of filter depth by the Artinianness of local cohomology modules  as follows.

\begin{lemma}\label{L41} {\rm (\cite[Theorem 3.1]{Me1})} 
$\f-depth(I, M)=\inf\{ i\mid H^i_I(M)\text{ is not Artinian }\}.$
\end{lemma}

Now, we apply Theorems \ref{KL} and \ref{T41} to prove the statement (ii) of Theorem \ref{MTh}. We will do this in the two following theorems corresponding to $N_n=M_n$ in the graded case $\fM=\oplus_{n\ge 0}M_n$ and $N_n= M/J^nM$.

\begin{theorem}\label{T42} Let $r_0$ be the eventual value of $\f-depth(I,M_n)$. Then the set 
$\qquad$
$\underset {j\le r_0}\bigcup\Ass_R(H^j_I(M_n))$ is stable for large $n$.
\end{theorem}

\begin{proof} By virtue of Lemmas \ref{L2} and \ref{L41} the theorem is trivial for the case of either $r_0=0$ or $r=\infty.$ Therefore it remains to prove the theorem in the case $1\le r_0<\infty$. Then by Theorem \ref{L3} and Lemmas \ref{L2} and \ref{L4} there exists an integer $a$ such that for all $n\ge a$ the following conditions are true: (i) $r_0=\f-depth(I,M_n)$; (ii) there is a sequence $x_1,\ldots,x_{r_0}\in I$ which is at the same time a permutable filter regular sequence of $M_n$ and a permutable $I-$filter regular sequence of $M_n$; and (iii) the set $\Ass_R(M_n/(x_1,\ldots,x_{r_0})M_n)$ is independent of $n$. We first prove that $\bigcup_{n\ge a}\Ass_R(H^{r_0}_I(M_n))$ is a finite set. Let $n\ge a$. We have $H^{r_0}_I(M_n)\cong H^0_I(H^{r_0}_{(x_1,\ldots,x_{r_0})}(M_n)$ by Lemma \ref{LL1}, and 
$$H^{r_0}_{(x_1,\ldots,x_{r_0})}(M_n)\cong\underset{t}{\varinjlim}(M_n/(x_1^t,\ldots,x_{r_0}^t)M_n)$$
by \cite[Theorem 5.2.9]{BS}. Thus  by virtue of   \cite[Proposition 2.6]{BN} we obtain 
$$\Ass_R(H^{r_0}_I(M_n))\subseteq\bigcup_{t>
0}\Ass_R(M_n/(x_1^t,\ldots,x_{r_0}^t)M_n)=\Ass_R(M_n/(x_1,\ldots,x_{r_0})M_n).$$
 Therefore the set $\bigcup_{n\geq a}\Ass_R(H^{r_0}_I(M_n))$ is finite by the choice of $a$. From this we may assume without loss of generality  that $a$ large enough such that the set $S=\bigcup_{n\ge a}\Ass_R(H^{r_0}_I(M_n))$ only  consists  of all prime ideals $\p$ which belong to $\Ass_R(H^{r_0}_I(M_n))$ for infinitely many $n\ge a$.\\ 
Next,  we claim that $\Ass_R(H^{r_0}_I(M_n))\setminus\{\m\}$ is stable for large $n$. Indeed, for any $\p\in S$ there is an infinite set $T$ of integers such that $\p\in\Ass_R(H^{r_0}_I(M_n))$ for all $n\in T$. Hence $H^{r_0}_{I_\p}((M_n)_\p)\not=0$ so that $\depth(I_\p,(M_n)_\p)\le r_0$. As  $\p\in\Supp(M_n/IM_n)\setminus\{\m\}$, $\depth(I_\p,(M_n)_\p)\geq r_0$ by Lemma \ref{LL0}. Thus  $\depth(I_\p,(M_n)_\p)=r_0$ for all $n\in T.$ Keep in mind that $\depth(I_\p,(M_n)_\p)$ is stable for  large $n$ by Theorem \ref{L3}. Hence $\depth(I_\p,(M_n)_\p)=r_0$ for all  large $n$. It follows by Theorem 
\ref{KL} that $\Ass_{R_\p}(H^{r_0}_{I_\p}((M_n)_\p))$ is stable for large $n$. Therefore $\Ass_R(H^{r_0}_I(M_n))\setminus\{\m\}$ is stable for large $n$, and the claim follows.\\
Let $r$ be the eventual value of $\depth(I,M_n)$. It is clear that $r\leq r_0$. If $r=r_0$, then $\bigcup_{j\le r_0}\Ass_R(H^j_I(M_n))=\Ass_R(H^{r_0}_I(M_n))$, and  this set is stable for large $n$ by Theorem \ref{KL}. If $r<r_0$, there exists an enough large $n$ such that  $H^{r}_I(M_n)\not=0$, which is an Artinian module  by Lemma \ref{L41}. It follows that
$$\bigcup_{j\leq r_0}\Ass_R(H^j_I(M_n))=\Ass_R(H^{r_0}_I(M_n))\cup\{\m\}$$
is stable for large $n$, and the theorem is proved. 
\end{proof}

Before proving the statement (ii) of Theorem \ref{MTh} for $M/J^nM$, we need the following lemma.

\begin{lemma}\label{L42} Let $A$ be a submodule of an $R$-module $K$. Then we have 
$$\Ass_R(K/A)\setminus\Supp(A)=\Ass_R(K)\setminus\Supp(A).$$
\end{lemma}
\begin{proof} Let $\p\in\Ass_R(K/A)\setminus\Supp(A)$ then $\p=(A:\nu)_R$ for
some $\nu\in K$. It yields that $\p \nu\subseteq A$. So as $A_\p=0$
we have $(\p \nu)_\p=0$. From this, since $\p \nu$ is a finitely
generated $R$-module, it easy to verify that $\p\subseteq\ann(r\nu)$ for
some $r\notin\p$. For  any $x\in\ann(r\nu)$ then
$rx\in\ann(\nu)\subseteq (A:\nu)_R=\p$; so $x\in\p$, since  $r\notin\p$.
Thus $\p=\ann(r\nu)$, so that $\p\in\Ass_R(K)$. It follows that $\Ass_R(K/A)\setminus\Supp(A)\subseteq\Ass_R(K)\setminus\Supp(A).$
The converse inclusion is trivial.
\end{proof}

\begin{theorem}\label{T43} Let $s_0$ be the eventual value of $\f-depth(I,M/J^nM)$. Then the set $\qquad$
$\bigcup_{j\leq s_0}\Ass_R(H^j_I(M/J^nM))$ is stable for large $n.$
\end{theorem}

\begin{proof} By Theorem \ref{L3}, there is an integer $a$ such that $r_0=\f-depth(I,J^nM/J^{n+1}M)$ and $s_0=\f-depth(I,M/J^nM)$ for all $n\ge a$. By Theorem \ref{T41} the result is trivial for the case  $s_0=\infty.$ Assume now that  $s_0<\infty$. Using the last part in the proof of Theorem \ref {T42} we need only to show that the set $\Ass_R(H^{s_0}_I(M/J^nM))\setminus\{\m\}$ is stable for large $n$. From the short exact sequence $0\sr J^nM/J^{n+1}M\sr M/J^{n+1}M\sr M/J^nM\sr 0,$ we get the following long exact
sequence of local cohomology modules
$$
\begin{aligned}
\cdots\sr H^{j-1}_I(M/J^nM)\sr H^j_I&(J^nM/J^{n+1}M)\\
&\sr H^j_I(M/J^{n+1}M)\sr H^j_I(M/J^nM)\sr\cdots
\end{aligned}
$$
for all $n>0$.  Note that $r_0\ge s_0$ by Lemma \ref{C22}. Thus we  consider two cases as follows.
\medskip

\noindent{\it Case 1: }$r_0-1\ge s_0$. By the long exact sequence, we get the following exact sequences 
$$H^{s_0}_I(J^nM/J^{n+1}M)\sr H^{s_0}_I(M/J^{n+1}M)\sr H^{s_0}_I(M/J^nM),$$
for all $n\ge a$, in which $H^{s_0}_I(J^nM/J^{n+1}M)$ is Artinian by Lemma \ref{L41}. Thus we have 
$$\Ass_R(H^{s_0}_I(M/J^{n+1}M))\setminus\{\m\}\subseteq\Ass_R(H^{s_0}_I(M/J^nM))\setminus\{\m\}$$
for all $n\ge a$. Since $\Ass_R(H^{s_0}_I(M/J^nM))$ is finite for $n=a$ by \cite[Theorem 2.5]{CH}, the set
$\Ass_R(H^{s_0}_I(M/J^nM))\setminus\{\m\}$ is stable for large $n$.
\medskip

\noindent{\it Case 2: }$r_0=s_0$. Note by Theorem \ref{T42} that  $\Ass_R(H^{s_0}_I(J^nM/J^{n+1}M))\setminus\{\m\}$ is stable for all $n\ge b$ for some integer $b\ge a$. Set $X=\Ass_R(H^{s_0}_I(J^bM/J^{b+1}M))\setminus\{\m\}.$ Thus we get by the exact sequence  
$$H^{s_0-1}_I(M/J^nM)\sr H^{s_0}_I(J^nM/J^{n+1}M)\sr H^{s_0}_I(M/J^{n+1}M)\sr H^{s_0}_I(M/J^nM)$$
for all $n\ge a$, in which $H^{s_0-1}_I(M/J^nM)$ is Artinian by Lemma \ref{L41}, and  Lemma \ref{L42} that 
$$X\subseteq \Ass_R(H_I^{s_0}(M/J^{n+1}M))\setminus\{\m\} \subseteq
(\Ass_R (H_I^{s_0}(M/J^nM))\setminus\{\m\})\cup X$$
for all $n\ge b$. Therefore
\begin{alignat}{2}
\Ass_R (H_I^{s_0}(M/J^{n+2}M))\setminus\{\m\}&\subseteq\big(\Ass_R(H_I^{s_0}(M/J^{n+1}M))\setminus\{\m\}\big)\cup X\notag\\
&=\Ass_R(H_I^{s_0}(M/J^{n+1}M))\setminus\{\m\}\notag
\end{alignat}
 for all $n\ge b$. Note that $\Ass_R(H_I^{s_0}(M/J^{n+1}M))\setminus\{\m\}$ is a finite set for $n=b$ by \cite[Theorem 2.5]{CH}. It follows that $\Ass_R(H_I^{s_0}(M/J^nM))\setminus\{\m\}$ is stable for large $n$ as required.
\end{proof}

\section{Proof of Theorem \ref{MTh}, (iii)}

To prove the statement (iii) of Theorem \ref{MTh} we first need the following known results.

\begin{lemma}\label{L51} {\rm (\cite[Corollary 4.4]{CH1})} The following equality holds true
$$\gdepth(I, M)=\inf\{ i\mid \Supp(H^i_I(M))\text{ is not finite }\}.$$
\end{lemma}

\begin{lemma}\label{L52} {\rm (\cite[Corollary 4.6]{CH1})} The set $\Ass(H^j_I(N))$ is finite for all $j\le\gdepth(I,N)$.
\end{lemma}

The first conclution of  Theorem \ref{MTh}, (iii) is a special case of the following result.
\begin{theorem}\label{T52} Let $r_1$ be the eventual value of $\gdepth(I,M_n)$. Then for each $l\le r_1$ the set  
$\bigcup_{j\le l}\Ass_R(H^j_I(M_n))\cup\{\m\}$ is stable for large $n.$
\end{theorem}

\begin{proof} By theorem \ref{L3} and Lemmas \ref{L2} and \ref{L4} we can choose a positive integer $a$ such that $d=\dim(M_n),\ d'=\dim(M_n/IM_n)\text{ and }r_1=\gdepth(I,M_n)$ for all $n\ge a$.  
We first claim that there is an integer $b\ge a$ such that $S=\bigcup_{n\ge b}\bigcup_{j\le l}\Ass_R(H^j_I(M_n))$ is finite. In order to do this we consider three cases: $r_1=0$, $r_1=\infty$ and $1\le r_1<\infty$. If $r_1=0$, then  $S\subseteq\bigcup_{n\ge a}\Ass_R(M_n)$ is finite by Lemma \ref{L2}. If $r_1=\infty$,  then $d'\le 1$, so that $S\subseteq\bigcup_{n\ge a}\Supp_R(M_n/IM_n)$ is finite by Lemma \ref{L2}.   Now we consider the case $1\le r_1<\infty$. Then, by Lemma \ref{L4} we can choose an integer $b\ge a$ such that the following statements are true for all $n\ge b$: (i) there is a sequence $x_1,\ldots,x_{r_1}\in I$ which is at the same time a permutable generalized regular sequence of $M_n$ and a permutable $I-$filter regular sequence of $M_n$; and (ii) the sets 
$\Ass_R(M_n/(y_1,\ldots,y_j)M_n)$ are stable for all $j\le r_1$. From this and  \cite[Proposition 2.6]{BN} we have 
$S\subseteq\bigcup_{n\ge b}\bigcup_{j\le l}\Ass_R(M_n/(y_1,\ldots,y_j)M_n).$
Therefore $S$ is finite by Lemma \ref{L2} again, and the claim is proved. \\
Since $S$ is finite, we can choose the integer $b$  enough large such that $S$ only consists of all  prime ideals $\p$, which belong to  $\bigcup_{j\le l}\Ass_R(H^j_I(M_n))$ for infinitely many $n\ge b$. For any $\p\in S_{\ge 1}$, by Theorem \ref{L3}, there is an integer $n(\p)\ge b$ such that $r(\p)=\depth(I_\p,(M_n)_\p)$ and $r_0(\p)=\f-depth(I_\p,(M_n)_\p)$ for all $n\ge n(\p)$. Then, by the choice of $S$ and Lemma \ref{LL0}, we obtain that $r(\p)\le l\le r_1\le r_0(\p)$. If $l<r_0(\p)$, $\bigcup_{j\le l}\Ass_{R_\p}(H^j_{I_\p}((M_n)_\p))=\{\p R_\p\}$ for all $n\ge n(\p)$ by Lemma \ref{L41}. If $l=r_0(\p)$,  it follows from Theorem \ref{T42} that $\bigcup_{j\le l}\Ass_{R_\p}(H^j_{I_\p}((M_n)_\p))$ is stable for all $n\ge m(\p)$ for some integer $m(\p)\ge b$. Therefore $\bigcup_{j\le l}\Ass_R(H^j_I(M_n))\cup\{\m\}$ is stable for all $n\ge\max\{n(\p),\ m(\p)\mid\p\in S_{\ge 1}\}$ as required.
\end{proof}

Finally, by applying Theorem \ref{T52} we can prove the statement (iii) of Theorem \ref{MTh} for the $R-$modules $M/J^nM$.
 
\begin{theorem}\label{T53} Let $s_1$ be the eventual value of $\gdepth(I,M/J^nM)$. Then for each $l\le s_1$ the set $\bigcup_{j\le l}\Ass_R(H^j_I(M/J^nM))\cup\{\m\}$ is stable for large $n$.
\end{theorem}

\begin{proof} By Theorem \ref{L3}, there is an integer $a$ such that $r_1=\gdepth(I,J^nM/J^{n+1}M)$ and $s_1=\gdepth(I,M/J^nM)$ for all $n\ge a$. Note that $s_1\le r_1$ by Lemma \ref{C22}. For any $0\le l\le s_1$ we set
$$ X_l(n)=\bigcup_{j\le l}\Ass_R(H^j_I(M/J^nM))\cup\{\m\} \text { and } 
S_l(n)=\bigcup_{j\le l}\Ass_R(H^j_I(J^nM/J^{n+1}M))\cup\{\m\}.$$
We will show  by induction on $l$ that the set $X_l(n)$ is stable for  large $n$. It is nothing to prove for the case $l=0$. Let $l>0$. 
From the short exact sequence $$0\sr J^nM/J^{n+1}M\sr M/J^{n+1}M\sr M/J^nM\sr 0$$ we get the following long exact
sequence
$$
\begin{aligned}
\cdots\sr H^{j-1}_I(M/J^nM)\sr H^j_I&(J^nM/J^{n+1}M)\\
&\xrightarrow{f_j} H^j_I(M/J^{n+1}M)\sr H^j_I(M/J^nM)\sr\cdots
\end{aligned}
$$
for all $n>0$. From this, for any $j\le l$ and $n\ge a$ we have
$$\Ass_R(H^j_I(M/J^{n+1}M))\subseteq\Ass_R(H^j_I(M/J^nM))\cup\Ass_R(\Im(f_j)).$$
Note that 
$$\Ass_R(\Im(f_j))\subseteq\Ass_R(H^j_I(J^nM/J^{n+1}M))\cup\Supp_R(\Ker(f_j))$$
by Lemma \ref{L42}. Since $j-1<s_1$,   $\Supp_R(H^{j-1}_I(M/J^nM))$ is a finite set by Lemma \ref{L51}, so that                                                                          $$\Supp_R(\Ker(f_j))\subseteq\Supp_R(H^{j-1}_I(M/J^nM))\subseteq\Ass_R(H^{j-1}_I(M/J^nM))\cup\{\m\}.$$ Therefore we get
$$X_l(n+1)\subseteq X_l(n)\cup S_l(n).$$ 
On the other hand, by the long exact sequence above and Lemma \ref{L51}, we have
$$
\begin{aligned}
\Ass_R(H^j_I(J^nM/J^{n+1}M))&\subseteq\Ass_R(H^j_I(M/J^{n+1}M))\cup\Supp_R(H^{j-1}_I(M/J^nM))\\
&\subseteq\Ass_R(H^j_I(M/J^{n+1}M))\cup\Ass_R(H^{j-1}_I(M/J^nM))\cup\{\m\}.
\end{aligned}
$$
 Hence $$S_l(n)\subseteq X_l(n+1)\cup X_{l-1}(n)$$  for all $n\ge a$.
By Theorem \ref{T52} and the inductive hypothesis, there exists an integer $b\ge a$ such that $S_l(n)=S$ and $X_{l-1}(n)=X$ for all $n\ge b$. Therefore, since $X=X_{l-1}(n+1)\subseteq X_l(n+1)$,
$$S\subseteq X_l(n+1)\cup X=X_l(n+1)$$
for all $n\ge b$. Finally, we obtain the following inclusions
$$X_l(n+2)\subseteq X_l(n+1)\cup S\subseteq X_l(n+1)$$
for all $n\ge b$. Thus  $X(n)$ is stable for large $n$, since $X_l(b+1)$ is finite by  Lemma \ref{L52}, and the proof of Theorem \ref {T53} is complete.
\end{proof}

\end{document}